\documentclass[12pt]{amsart}

\usepackage{amssymb}

\usepackage{enumerate}

\usepackage{graphicx}

\makeatletter
\@namedef{subjclassname@2010}{%
  \textup{2010} Mathematics Subject Classification}
\makeatother

\newtheorem{theorem}{Theorem}[section]

\newtheorem{lemma}[theorem]{Lemma}

\newtheorem{proposition}[theorem]{Proposition}
\theoremstyle{definition}

\newtheorem{remark}[theorem]{Remark}

\numberwithin{equation}{section}

\begin{document}

%%%%% To ease editing, for IMPAN journals add:

\baselineskip=17pt

%%%%%%%%%%%

%% In the running head, replace first names by initials 
%% and give an abbreviation of the title.

\title[Estimation of Szlenk index]{Estimation of the Szlenk index \\ of reflexive Banach spaces using \\ Generalized Baernstein spaces}

\author[R. Causey]{Ryan Causey}
\address{Ryan Causey \\ Department of Mathematics \\ Texas A\&M University \\
College Station, TX 77845}
\email{rcausey@math.tamu.edu}

\date{}

\begin{abstract}For each ordinal $\alpha<\omega_1$, we prove the existence of a separable, reflexive Banach space with a basis and Szlenk index $\omega^{\alpha+1}$ which is universal for the class of separable, reflexive Banach spaces with Szlenk index not exceeding $\omega^\alpha$.  

\end{abstract}



\maketitle

\section{Introduction}

The relatively new tool of weakly null trees has produced a number of recent results in Banach space theory.  In particular, trees have facilitated the solution of coordinitization questions through a strong connection between trees and embedding into Banach spaces with an FDD which has prescribed properties.  For example, Johnson and Zheng completely characterized when a separable reflexive space embeds into a reflexive space with unconditional basis \cite{JZh} and when a separable Banach space embeds into a Banach space with shrinking, unconditional basis \cite{JZh1} using the UTP and $w^*$UTP, respectively.  Odell and Schlumprecht demonstrated that for $1< p< \infty$, a separable, reflexive space embeds into a Banach space which is the $\ell_p$ sum of finite dimensional spaces if and only if every normalized, weakly null tree has a branch equivalent to the $\ell_p$ unit vector basis \cite{OS1}.  In a spirit which we continue, Odell and Schlumprecht established a strong connection between tree estimates and embeddings into Banach spaces with the corresponding block estimates (the relevant notions are defined in Section $2$).  These coordinitization results provide an avenue for the proof of the existence of universal Banach spaces for classes of spaces with certain tree estimates.  

Our results follow the methods of Odell, Schlumprecht, and Zs\'{a}k \cite{OSZ2}, and Freeman, Odell, Schlumprecht, and Zs\'{a}k \cite{FOSZ} who used Tsirelson spaces in their constructions.  In \cite{FOSZ}, the objects of study were Banach spaces with separable dual, while in \cite{OSZ2}, the objects were separable, reflexive spaces.  The former proved both a coordinitization result and a universality result concerning the classes of separable Banach spaces with Szlenk index not exceeding $\omega^{\alpha \omega}$, while the latter proved a coordinitization result and a universality result concerning the classes of separable, reflexive Banach spaces $X$ such that the Szlenk indices of both $X$ and $X^*$ do not exceed $\omega^{\alpha\omega}$.  In \cite{CA}, results analogous to those of \cite{FOSZ} were established using Schreier spaces.  These results allowed finer gradations by working instead with the classes of separable Banach spaces with Szlenk index not exceeding $\omega^\alpha$ for a countable ordinal $\alpha$.  Two-sided estimates were not possible with Schreier spaces, which are $c_0$-saturated.  To establish two-sided estimates, we introduce a generalization of the so-called Baernstein space, which we denote $X^p_\alpha$, which is itself a generalization of Schreier's original space.  The details of the construction are given in Section $3$.  This allows us to improve the results of \cite{OSZ2} by making finer gradations, as in \cite{CA}.  In Section $4$, we define the Szlenk index, originally used by Szlenk to prove the non-existence of a separable, reflexive Banach space which is universal for the class of separable, reflexive Banach spaces.  We also recall several results concerning the Szlenk index which connect it to tree estimates.  Last, we present the proofs of the main results in Section $5$.  Our connection between the Szlenk index and tree estimates is summarized in 

\begin{theorem} If $X$ is a separable, reflexive Banach space, $\alpha<\omega_1$ is such that $Sz(X), Sz(X^*)\leq \omega^\alpha$, and $1<p\leq 2$, then $X$ satisfies subsequential $((X_\alpha^p)^*, X_\alpha^p)$-tree estimates. \label{name1}\end{theorem}

A major idea behind Theorem \ref{name1} is the comparison of normalized block sequences in two Banach spaces to make a comparison of the Szlenk indices of the two spaces.  To make this comparison, we establish the following coordinitization result, which connects tree estimates with block estimates.  

\begin{theorem} Let $1<p\leq 2$ and let $U$ be a reflexive Banach space with a normalized, $1$-unconditional, right dominant basis satisfying subsequential $U^p$-upper block estimates in $U^p$.  If $X$ is a separable, reflexive Banach space satisfying subsequential $((U^p)^*,U^p)$-tree estimates, then $X$ embeds into a Banach space $Z$ with FDD $E$ satisfying subsequential $((U^p)^*,U^p)$-block estimates in $Z$.  \end{theorem}

We last employ a theorem of Johnson, Rosenthal, and Zippin from \cite{JRZ} to deduce the existence of a universal space with a basis.  For this, we define for each $\alpha<\omega_1$ the class $$\mathcal{C}_\alpha=\{X: X \text{\ separable, reflexive}, Sz(X), Sz(X^*)\leq \omega^\alpha\}.$$  

\begin{theorem} Let $\alpha<\omega_1$.  There exists a separable, reflexive space $W\in \mathcal{C}_{\alpha+1}$ with a basis which is universal for the class $\mathcal{C}_\alpha$.  \label{name3}\end{theorem}

This is a strengthening of Theorem $C$ of \cite{OSZ2}, which proved the above result in the case that $\alpha=\beta\omega$ for some $\beta<\omega_1$.  

This paper was completed at Texas A\&M under the direction of Thomas Schlumprecht as part of the author's doctoral dissertation.  The author thanks Dr. Schlumprecht for his insights and direction during its completion. 

\section{Definitions and Notation}

If $Z$ is a Banach space and $E=(E_n)$ is a collection of finite-dimensional normed spaces, we say $E$ is a \emph{finite dimensional decomposition}, or \emph{FDD}, for $Z$ if for each $z\in Z$ there exists a unique sequence $(z_n)$ such that $z_n\in E_n$ and $z=\sum_{n=1}^\infty z_n$.  If $E$ is an FDD for a Banach space $Z$, for $n\in \mathbb{N}$ we denote the $n^{th}$ coordinate projection by $P^E_n$.  More precisely, for $z\in Z$, if $z=\sum_{n=1}^\infty z_n$ for $z_n\in E_n$, $P^E_nz=z_n$.  For a finite $A\subset \mathbb{N}$, we let $P_A^E=\sum_{n\in A}P^E_n$.  We define the \emph{projection constant} $K=K(E,Z)$ to be $$K=K(E,Z)=\underset{m\leq n}{\sup}\|P^E_{[m,n]}\|.$$  This is finite by the principle of uniform boundedness.  We call $E$ a \emph{bimonotone} FDD for a Banach space $Z$ if $K(E,Z)=1$.  If $Z$ has an FDD, we can always endow $Z$ with an equivalent norm which makes $E$ a bimonotone FDD.  We let $\text{supp}_E z= \{n: P_n^E z \neq 0\}$, and call this set the \emph{support} of $z$.  If $E$ is a basis, or if no confusion is possible, we write $\text{supp\ }z$ in place of $\text{supp}_E z$.  We denote by $c_{00}\Bigl(\oplus E_n\Bigr)$ the collection of vectors in $z$ with finite support.  We note that $c_{00}\Bigl(\oplus E_n\Bigr)$ is dense in any space for which $E$ is an FDD.  

We denote by $Z^{(*)}$ the closed span of $c_{00}\Bigl(\oplus E_n^*\Bigr)$ in $Z^*$ and note that $E^*=(E_n^*)$ is an FDD for $Z^{(*)}$ with $K(E^*, Z^*)\leq K(E,Z)$.  We consider $E^*_n$ with the norm it inherits as a subspace of $Z^*$ and not with the norm it inherits as the dual of $E_n$.  These norms may be different if $E$ is not bimonotone.  If $Z^{(*)}=Z^*$, we say that $E$ is a \emph{shrinking} FDD for $Z$.  We say that $E$ is a \emph{boundedly-complete} FDD if for each sequence $(z_n)$ with $z_n\in E_n$ such that $\underset{N\in \mathbb{N}}{\sup}\Bigl\|\sum_{n=1}^N z_n\Bigr\|<\infty$, $\sum_{n=1}^\infty z_n$ converges in $Z$.  A Banach space $Z$ with FDD $E$ is reflexive if and only if the FDD is both shrinking and boundedly-complete.  

If $Z$ is a Banach space with FDD $E=(E_n)$ and $V$ is a Banach space with a normalized, $1$-unconditional basis $(v_n)$, we define the space $Z^V=Z^V(E)$ to be the completion of $c_{00}\Bigl(\underset{n=1}{\overset{\infty}{\oplus}} E_n\Bigr)$ endowed with the norm $$\|z\|_{Z^V} = \sup\Bigl\{\Bigl\| \sum_{k=1}^\infty \|P^E_{[k_{n-1}, k_n)}z\|_Z v_{k_{n-1}}\Bigr\|_V: 1\leq k_0<k_1<\ldots\Bigr\}.$$  The norm above depends upon the FDD $E$, but when no confusion is possible, we will write $Z^V$ in place of $Z^V(E)$.  For convenience, we will write $Z^p$ in place of $Z^{\ell_p}$.  

If $U$ is a Banach space and $(u_n)$ is a basis for $U$, we say $(u_n)$ is $R$-right dominant if for each pair of subsequences of the natural numbers $(m_n), (k_n)$ with $m_n\leq k_n$ for all $n$, $(u_{m_n})$ is $R$-dominated by $(u_{k_n})$.  If $B=(b_n)$ is a subsequence of the natural numbers, we let $U_B=[u_{b_n}]$.  
If $Z$ is a Banach space with FDD $E$, and $U$ is a Banach space with a normalized, $1$-unconditional basis $(u_n)$, we say $E$ satisfies \emph{subsequential} $C$-$U$-\emph{upper }(respectively, \emph{lower}) \emph{block estimates} in $Z$ if each normalized block sequence $(z_n)$ is $C$-dominated by (respectively, $C$-dominates) $(u_{m_n})$, where $m_n=\min \text{supp}_E z_n$.  We say $X$ satisfies \emph{subsequential} $K$-$(U^{(*)},U)$-\emph{block estimates} in $Z$ if it satisfies $K$-$U^{(*)}$-lower block estimates in $Z$ and $K$-$U$-upper block estimates in $Z$. 

We next recall a coordinate free version of subsequential upper and lower estimates.  For $\ell\in \mathbb{N}$, we define $$T_\ell=\{(n_1,\ldots,n_\ell): n_1<\ldots<n_\ell, n_i\in \mathbb{N}\}$$

and $$T_\infty=\underset{\ell=1}{\overset{\infty}{\bigcup}}T_\ell,\text{\ \ \ \ \ \ \ } T_\infty^\text{even}=\underset{\ell=1}{\overset{\infty}{\bigcup}}T_{2\ell}.$$

An \emph{even tree} in a Banach space $X$ is a family $(x_t)_{t\in T^\text{even}_\infty}$ in $X$.  Sequences of the form $(x_{(t,k)})_{k>k_{2n-1}}$, where $n\in \mathbb{N}$ and $t=(k_1,\ldots,k_{2n-1})\in T_\infty$, are called \emph{nodes}.  A sequence of the form $(k_{2n-1},x_{(k_1,\ldots,k_{2n})})_{n=1}^\infty$, with $k_1<k_2<\ldots$, is called a \emph{branch} of the tree.  An even tree is called \emph{weakly null} if every node is a weakly null sequence.  If $X$ is a dual space, an even tree is called $w^*$ \emph{null} if every node is $w^*$ null.  If $X$ has an FDD $E=(E_n)$, a tree is called a \emph{block even tree of} $E$ if every node is a block sequence of $E$.  

If $T\subset T^\text{even}_\infty$ is closed under taking restrictions so that for each $t\in T\cup \{\varnothing\}$ and for each $m\in \mathbb{N}$ the set $\{n\in \mathbb{N}: (t,m,n)\in T\}$ is either empty or infinite, and if the latter occurs for infinitely many values of $m$, then we call $(x_t)_{t\in T}$ a \emph{full subtree}.  Such a tree can be relabeled to a family indexed by $T^\text{even}_\infty$ and the branches of $(x_t)_{t\in T}$ are branches of $(x_t)_{t\in T^\text{even}_\infty}$ and that the nodes of $(x_t)_{t\in T}$ are subsequences of the nodes of $(x_t)_{t\in T^\text{even}_\infty}$. 

Let $U$ be a Banach space with a normalized, $1$-unconditional basis $(u_n)$ and $C\geq 1$.  Let $X$ be an infinite-dimensional Banach space.  We say that $X$ \emph{satisfies subsequential} $C$-$U$-\emph{upper tree estimates} if every normalized, weakly null even tree $(x_t)_{t\in T^\text{even}_\infty}$ in $X$ has a branch $(k_{2n-1},x_{(k_1, \ldots,k_{2n})})$ such that $(x_{(k_1, \ldots, k_{2n})})_n$ is $C$-dominated by $(u_{k_{2n-1}})_n$.  We say $X$ \emph{satisfies subsequential} $C$-$U$-\emph{lower tree estimates} if every normalized, weakly null even tree $(x_t)_{t\in T^\text{even}_\infty}$ in $X$ has a branch $(k_{2n-1},x_{(k_1,\ldots,k_{2n})})$ so that $(x_{(k_1, \ldots, k_{2n})})_n$ $C$-dominates $(v_{k_{2n-1}})$.  We say that $X$ \emph{satisfies subsequential} $U$-\emph{upper (respectively lower) tree estimates} if it satisfies $C$-$U$-upper (respectively lower) tree estimates for some $C\geq 1$.  We say $X$ satisfies subsequential $C$-$(U^{(*)},U)$-tree estimates if it satisfies subsequential $C$-$U^{(*)}$-lower tree estimates and $C$-$U$-upper tree estimates.  

We let $\mathcal{A}_{(U^{(*)},U)}$ denote the class of all separable, reflexive Banach spaces which satisfy subsequential $(U^{(*)},U)$-tree estimates.

A simple perturbation argument yields the following.  

\begin{lemma} Let $U$ be a Banach space with a normalized, $1$-unconditional basis $(u_n)$ and $Z$ is a Banach space with FDD $E=(E_n)$ satisfying subsequential $C$-$U$-upper (respectively, lower) block estimates in $Z$.  Assume also that for each $n\in \mathbb{N}$, $E_n\neq \{0\}$.  If $(z_n)$ is a normalized block sequence in $Z$ and $(k_n)\subset \mathbb{N}$ is strictly increasing such that $$k_n\leq \min \emph{supp}_E z_n\leq \max \emph{supp}_E z_n<k_{n+1},$$ then $(z_n)$ is $C$-dominated by (respectively, $C$-dominates) $(u_{k_n})$. \label{perturb} \end{lemma}

Another simple but technical lemma involves the preservation of upper block estimates.  We postpone the proof until Section $4$.  

\begin{lemma} Let $1\leq p\leq 2$, and let $B=(b_n),C=(c_n)$ be arbitrary subsequences of the natural numbers.  Let $U$ be a Banach space with a normalized, $1$-unconditional basis $(u_n)$.  Suppose $Z$ is a Banach space with bimonotone FDD $E$ which satisfies subsequential $U^p_B$-upper block estimates in $Z$.  Then $E$ satisfies subsequential $U^p_B$ upper block estimates in $Z^{((U^p_C)^{(*)})}$.  \label{upperest}\end{lemma}

\section{Schreier Families, Schreier and Baernstein Spaces}

Throughout, we will assume subsets of $\mathbb{N}$ are written in increasing order.  Let $[\mathbb{N}]^{<\omega}$ denote the set of all finite subsets of $\mathbb{N}$, and $[\mathbb{N}]^\omega$ the set of all infinite subsets of $\mathbb{N}$.  We write $E<F$ if $\max E< \min F$.  By convention, $\min \varnothing=\infty$, $\max \varnothing =0$.  We consider the families $[\mathbb{N}]^\omega,[\mathbb{N}]^{<\omega}$ as being ordered by extension.  That is, the predecessors of an element are its initial segments, and we write $E\preceq F$ if $E$ is an initial segment of $F$.  A family $\mathcal{F}\subset [\mathbb{N}]^{<\omega}$ is called \emph{hereditary} if, whenever $E\in \mathcal{F}$ and $F\subset E$, $F\in \mathcal{F}$.  We associate a set $F$ with the function $1_F\in \{0,1\}^\mathbb{N}$, topologized with the product topology.  We then endow $[\mathbb{N}]^{<\omega}$ with the topology induced by this association.  We note that a hereditary family is compact if and only if it contains no strictly ascending chains.  

Given two (finite or infinite) subsequences $(k_n),(\ell_n)\subset \mathbb{N}$ of the same length, we say $(\ell_n)$ is a \emph{spread} of $(k_n)$ if $k_n\leq \ell_n$ for all $n\in \mathbb{N}$.  We call a family $\mathcal{F}\subset [\mathbb{N}]^{<\omega}$ \emph{spreading} if it contains all spreads of its elements. 

We construct the \emph{Schreier families} with more specific properties than is usually done.  Let $$\mathcal{S}_0=\{\varnothing\}\cup \bigl\{\{n\}: n\in \mathbb{N}\bigr\}.$$  Assuming that for $\alpha<\omega_1$, $\mathcal{S}_\alpha$ has been defined, let $$\mathcal{S}_{\alpha+1}=\bigl\{\underset{k=1}{\overset{n}{\bigcup}}E_k: E_k\in \mathcal{S}_\alpha, E_1<\ldots<E_n, n\leq \min E_1\bigr\}.$$  Assume $\alpha<\omega_1$ is a limit ordinal.  Assume also that for each $0\leq \beta<\alpha$, $\mathcal{S}_\beta$ has been defined, and for each limit ordinal $\lambda<\alpha$, there exists a sequence $\lambda_n\uparrow \lambda$ such that $\mathcal{S}_\lambda=\{E:\exists n\leq \min E, E \in \mathcal{S}_{\lambda_n+1}\}$.  An easy induction argument shows that for any $\beta<\gamma<\alpha$ there exists a non-negative integer $m$ such that $\mathcal{S}_\beta\subset \mathcal{S}_{\gamma+m}$.  Choose some sequence $\beta_n\uparrow \alpha$.  We can recursively choose non-negative integers $m_n$ so that $$\mathcal{S}_{\beta_n+m_n+1}\subset \mathcal{S}_{\beta_{n+1}+m_{n+1}}.$$  We let $\alpha_n=\beta_n+m_n$, so $\alpha_n\uparrow \alpha$.  Therefore we have the containment $\mathcal{S}_{\alpha_n+1}\subset \mathcal{S}_{\alpha_{n+1}}$.  We let $$\mathcal{S}_\alpha=\{E: \exists n\leq \min E, E\in \mathcal{S}_{\alpha_n+1}\}.$$  

The families above depend on the choices we make of $\alpha_n\uparrow \alpha$ for limit ordinals $\alpha$, but it is known that regardless of these choices, $\mathcal{S}_\alpha$ is spreading, hereditary, and compact.

\vspace{5mm}

Next, we recall the Repeated Averages Hierarchy as defined in \cite{AMT}.  For a partially ordered set $P$, we write $MAX(P)$ to denote the collection of maximal elements.  For each $I\in [\mathbb{N}]^\omega$, $0\leq \alpha<\omega_1$, we define a sequence $(x^{\alpha, I}_n)_n$ to have the properties \begin{enumerate}[(i)] \item $(x^{\alpha, I}_n)_n$ is a convex blocking of $(e_{i_n})$, \item $I=\underset{n=1}{\overset{\infty}{\bigcup}}\text{supp\ }x^{\alpha, I}_n$, \item $\text{supp\ }x^{\alpha, I}_n\in MAX(\mathcal{S}_\alpha)$ for each $n$.  \end{enumerate}

For $I\in [\mathbb{N}]^\omega$, write $I=(i_n)$.  We let $x^{0,I}_n=e_{i_n}$.  If $(x^{\alpha,I}_n)$ has been defined to have the properties above, we let $$x^{\alpha+1,I}_1=i_1^{-1}\displaystyle\sum_{j=1}^{i_1} x^{\alpha,I}_j.$$

Suppose that $x^{\alpha+1,I}_n$ has been defined for $1\leq n <N$ to be a convex blocking of $(e_{i_n})$, $\underset{n=1}{\overset{N-1}{\bigcup}}\text{supp}(x_n^{\alpha+1,I})$ is an initial segment of $I$, $\text{supp}(x_n^{\alpha+1,I})\in \text{MAX}(S_{\alpha+1})$ for each $n$, and $x_n^{\alpha+1,I}=\frac{1}{s_n}\displaystyle\sum_{j=p_{n-1}+1}^{p_n}x_j^{\alpha,I}$ for some $0=p_0<\ldots<p_{N-1}$, where $s_n=\min \text{supp}(x_n^{\alpha+1,I})$.  Then let $s_N=\min \text{supp\ }(x_{p_{N-1}+1}^{\alpha,I})$, $p_N=p_{N-1}+s_N$, $$x_N^{\alpha+1,I}=\frac{1}{s_N}\displaystyle\sum_{j=p_{N-1}+1}^{p_N}x_j^{\alpha,I}.$$

Last, assume that for a limit ordinal $\alpha<\omega_1$, $(x^{\beta,I'}_n)$ has been defined for all $\beta<\alpha$, and all $I'\in [\mathbb{N}]^\omega$.  Let $\alpha_n\uparrow \alpha$ be the ordinals used to define $\mathcal{S}_\alpha$.  Let $m_1= \min I$ and $x^{\alpha,I}_1=x^{\alpha_{m_1}+1,I}_1$.  Given $x^{\alpha,I}_n$ for $1\leq n<N$ with the same assumptions as in the successor case, let $I_N=I\setminus \underset{n=1}{\overset{N-1}{\bigcup}}x_n^{\alpha,I}$, $m_N = \min I_N$, and $x^{\alpha,I}_N=x^{\alpha_{m_N}+1,I_N}_1$.

\vspace{5mm}
For our next lemma, we must define a convenient notation.  If $x\in c_{00}$ and $E\subset \mathbb{N}$, we let $Ex$ be the sequence defined by $Ex(n)=1_E(n)x_n$.  

\begin{lemma} If $I=(i_n)\in [\mathbb{N}]^\omega$ is such that $3i_n\leq i_{n+1}$ and $E\in S_\alpha$, then $$\Bigl\|E\Bigl(\displaystyle\sum_{n=1}^\infty x^{\alpha,I}_n\Bigr)\Bigr\|_1 \leq 2.$$

\label{c_0dom}\end{lemma}

\begin{proof} By induction.  Since $S_\alpha$ is hereditary for each $\alpha$, it suffices to consider $E\subset I$.  If $\alpha=0$, the claim is clear, since $\varnothing \neq E\in S_0$ means $E$ is a singleton, and $(x_n^{0,I})=(e_{i_n})$.  

Next, assume the claim holds for the ordinal $\alpha$.  Let $E=\underset{k=1}{\overset{m}{\bigcup}}E_k\in S_{\alpha+1}$, $E_k\in S_\alpha$.  Let $m_n=\min \text{supp}(x_n^{\alpha+1,I})$.  If the set $\{n: \text{supp}(x_n^{\alpha+1,I})\cap E\neq \varnothing\}$ is empty, then the claim is trivial.  Suppose this set is non-empty, and let $N$ be its minimum.  Then $m\leq \min E\leq \frac{m_{N+1}}{3}$, and, inductively, $m\leq \frac{m_{N+n}}{3^n}$ for each $n\geq 1$.  Since there exists a sequence $0=p_0<p_1<\ldots$ with $$\text{supp}(x_n^{\alpha+1,I})=m_n^{-1}\displaystyle\sum_{j=p_{n-1}+1}^{p_n}x_j^{\alpha,I},$$
our inductive hypothesis gives that for each $j\leq m$, $$\Bigl\|E_jx_n^{\alpha+1,I}\Bigr\|_1\leq \frac{2}{m_n}.$$

Then \begin{align*} \Bigl\|E\Bigl(\displaystyle\sum_{n=1}^\infty x^{\alpha,I}_n\Bigr)\Bigr\|_1 & \leq \|x_N^{\alpha,I}\|_1+\displaystyle\sum_{n=1}^\infty\displaystyle\sum_{j=1}^m \|E_jx^{\alpha,I}_{N+n}\|_1 \\ & \leq 1+\displaystyle\sum_{n=1}^\infty \frac{2m}{m_{N+n}} \leq 1+2\displaystyle\sum_{n=1}^\infty 3^{-n} =2.\end{align*}

Last, let $\alpha<\omega_1$ be a limit ordinal and assume the claim holds for all $\beta<\alpha$.  Let $\alpha_n\uparrow \alpha$ be the ordinals used to define $\mathcal{S}_\alpha$.  If $E\in S_\alpha$, let $N=\min\{n: \text{supp}(x_n^{\alpha,I})\cap E\neq \varnothing\}$.  Let $m= \min E$, $m_n=\min \text{supp} x^{\alpha, I}_n$.  For each $n\geq 1$, $m< m_{N+n}$.  Since $E\in S_\alpha$, it follows that $E\in S_{\alpha_m+1}\subset S_{\alpha_{m_{N+n}}}$.  Since $$x^{\alpha, I}_{N+n} = x^{\alpha_{m_{N+n}}+1,I}_1 = m_{N+n}^{-1}\sum_{k=p_n}^{p_n+m_{N+n}} x^{\alpha_{m_{N+n}}, I}_k$$ for some $p_n$, the inductive hypothesis implies $$\|Ex^{\alpha,I}_{N+n}\|_1\leq 2/m_{N+n} \leq 2/3^n.$$  As in the successor ordinal case, \begin{align*} \Bigl\|E\Bigl(\displaystyle\sum_{n=1}^\infty x^{\alpha,I}_n\Bigr)\Bigr\|_1 & \leq \|x^{\alpha,I}_N\|_1+\displaystyle\sum_{n=1}^\infty \|Ex^{\alpha,I}_{N+n}\|_1 \leq  1+2\displaystyle\sum_{n=1}^\infty 3^{-n}=2.\end{align*}

\end{proof}

We last define the spaces which we will use to prove our theorems, as well as deduce some of their properties.  For $\alpha<\omega_1$, we define the norm $\|\cdot\|_\alpha$ on $c_{00}$ by $$\|x\|_\alpha= \underset{E\in S_\alpha}{\sup}\|Ex\|_1.$$ The completion of $c_{00}$ under this norm is known as the \emph{Schreier space of order} $\alpha$, and denoted $X_\alpha$.  We see that the canonical basis $(e_n)$ of $c_{00}$ becomes a normalized, $1$-unconditional basis for $X_\alpha$.  We note also that the canonical basis is shrinking in $X_\alpha$ (this follows, for example, from \cite{JO}, where it was shown that $X_\alpha$ contains no copy of $\ell_1$).  We will consider spaces of the form $X_\alpha^p = (X_\alpha)^{\ell_p}$, as defined in Section $2$.  The space $X_1^2$ was introduced by Baernstein, and the generalizations $X_1^p$ were studied by Seifert \cite{CS}.  For this reason, we refer here to $X^p_\alpha$ as the \emph{Baernstein space of order} $\alpha$ and \emph{parameter} $p$.    

We note that for $x\in c_{00}$, \begin{align*} \|x\|_{X^p_\alpha} & = \sup\Bigl\{\Bigl(\sum_j \bigl(\sum_{i\in E_j}|x_i|\bigr)^p\Bigr)^{1/p}: E_1<E_2<\ldots, E_j\in S_\alpha\Bigr\} \\ & = \sup \Bigl\{ \Bigl\|\bigl(\|E_j x\|_1\bigr)_j\Bigr\|_{\ell_p}: E_1<E_2<\ldots, E_j\in S_\alpha\Bigr\},\end{align*} with the appropriate modification to the first line if $p=\infty$.  The same is true if the suprema run over all finite sequences $E_1<\ldots <E_n, E_j\in S_\alpha$.  We collect some relevant facts about the unit vector basis $(e_n)$ of $X^p_\alpha$ in the following lemma.  

\begin{lemma} Fix $\alpha<\omega_1, 1<p<\infty$.  Then the unit vector basis $(e_n)$ of $X^p_\alpha$ is shrinking, boundedly-complete, right dominant, and satisfies subsequential $X^p_\alpha$-upper block estimates in $X^p_\alpha$.  \label{Baernsteinprop}\end{lemma}

\begin{proof} Since the unit vector basis of $X_\alpha$ is shrinking, it is shrinking and boundedly-complete in $X_\alpha^p$ by \cite[Lemma 8, Corollary 7]{OSZ1}.  Therefore we deduce $X^p_\alpha$ is reflexive and the coordinate functionals $(e_n^*)$ form a normalized, $1$-unconditional basis for $(X^p_\alpha)^*$.  

Take $(m_n), (k_n)$ so that $m_n\leq k_n$.  Fix $a_n\in c_{00}$.  Let $x=\sum a_n e_{m_n}, y=\sum a_n e_{k_n}$.  There exists a sequence $E_1<E_2<\ldots$ with $E_j\in S_\alpha$ for each $j$ so that $$\|x\|_{X_\alpha^p}^p=\sum_j \bigl(\sum_{i\in I_j}|a_i|\bigr)^p,$$ where $I_j = \{i:m_i\in E_j\}$.  Let $M_j = \{m_i: i\in I_j\}$.  Then $M_j\subset E_j$, and we can assume $M_j=E_j$.  Let $K_j=\{k_i:i\in I_j\}$.  Then $K_j$ is a spread of $M_j$, and thus $K_j\in S_\alpha$.  Clearly we also have $K_1<K_2<\ldots$, and $$ \|y\|_{X_\alpha^p}^p \geq \sum_j \bigl(\sum_{i\in I_j}|a_i|\bigr)^p = \|x\|_{X_\alpha^p}^p.$$  Therefore $(e_n)$ is $1$-right dominant in $X^p_\alpha$.  

Next, take $E_1<E_2<\ldots, E_j\in S_\alpha$, $(z_n)$ a normalized block sequence in $X^p_\alpha$ with $m_n = \min \text{supp\ } z_n$.  We can write $z_n = w_n+x_n+y_n$, where $(w_n), (x_n), (y_n)$ are subnormalized and such that the support of each each $w_n$ or $y_n$ intersects at most one $E_j$ and for each $j$ there exists at most one $n$ so that $E_j\cap \text{supp\ } x_n\neq \varnothing$.  Let $$J=\{j\in \mathbb{N}: E_j\cap \text{supp\ } z_n \neq \varnothing \text{\ for some\ }n\}.$$  By \cite[Proposition 3.1]{CA}, there exists a sequence of successive sets $(F_j)_{j\in J}$ so that $F_j\in S_\alpha$ for each $j\in J$ and $$\Bigl\|E_j\bigl(\sum a_nx_n\bigr)\Bigr\|_1 \leq 2\Bigl\|F_j \bigl(\sum a_n e_{m_n}\bigr)\Bigr\|_1.$$  This means \begin{align*}\Bigl(\sum_j \Bigl\|E_j\bigl(\sum_n a_nx_n\bigr)\Bigr\|_1^p\Bigr)^{1/p} & \leq 2  \Bigl(\sum_j \Bigl\|F_j\bigl(\sum_n a_ne_{m_n}\bigr)\Bigr\|_1^p\Bigr)^{1/p} \\ & \leq 2\Bigl\|\sum _{n=1}^\infty a_n e_{m_n}\Bigr\|_{X_\alpha^p}.\end{align*} 
Moreover, since for each $n$ there exists at most one $j_n$ such that $E_j\cap \text{supp\ } w_n\neq \varnothing$, and since the unit vector basis of $X^p_\alpha$ clearly $1$-dominates the unit vector basis of $\ell_p$, we deduce (the unindexed sums are taken over all $n$ such that there exists some $j$ with $E_j\cap \text{supp\ } w_n\neq \varnothing$) \begin{align*} \sum_j \Bigl\|E_j\bigl(\sum_k a_k w_k\bigr)\Bigr\|_1^p & =\sum\Bigl\| E_{j_n}\bigl(\sum_k a_k w_k\bigr)\Bigr\|_1^p = \sum |a_n|^p \|E_{j_n}w_n\|_1^p \\ & \leq \sum |a_n|^p \|w_n\|_{X_\alpha}^p \leq \sum |a_n|^p \\ & \leq \Bigl\|\sum_{n=1}^\infty a_n e_{m_n}\Bigr\|_{X_\alpha^p}^p. \end{align*}  
Similarly, we deduce $$\sum_j \Bigl\|E_j\bigl(\sum_n a_n y_n\bigr)\Bigr\|_1^p \leq \Bigl\|\sum_{n=1}^\infty a_n e_{m_n}\Bigr\|_{X^p_\alpha}^p.$$  Therefore \begin{align*} \Bigl(\sum_j \Bigl\|E_j\bigl(\sum a_n z_n\bigr)\Bigr\|_1^p\Bigr)^{1/p} & \leq \Bigl(\sum_j \Bigl\|E_j\bigl(\sum a_n w_n\bigr)\Bigr\|_1^p\Bigr)^{1/p} \\ & +\Bigl(\sum_j \Bigl\|E_j\bigl(\sum a_n x_n\bigr)\Bigr\|_1^p\Bigr)^{1/p} \\ & + \Bigl(\sum_j \Bigl\|E_j\bigl(\sum a_n y_n\bigr)\Bigr\|_1^p\Bigr)^{1/p} \\ & \leq 4\Bigl\|\sum a_n e_{m_n}\Bigr\|_{X_\alpha^p}.\end{align*}
Since $E_1<E_2<\ldots$ was arbitrary, we deduce $(e_n)$ satisfies subsequential $4$-$X^p_\alpha$-upper block estimates in $X^p_\alpha$.  

\end{proof}

We conclude this section with the following extension of \ref{c_0dom}.  

\begin{lemma} Fix $1\leq p <\infty$.  If $I=(i_n)\in [\mathbb{N}]^\omega$ is such that $i_{n+1}\geq 3i_n$, $\alpha<\omega_1$, and $(x^{\alpha, I}_n)$ is the sequence of repeated averages, then $(x^{\alpha,I}_n)$ as a sequence in $X^p_\alpha$ is $5$-equivalent to the unit vector basis of $\ell_p$.  \label{lpdom}\end{lemma}

\begin{proof} Since $\|x^{\alpha, I}_n\|_1=1$ and $\text{supp\ }x^{\alpha, I}_n\in S_\alpha$, the sequence of repeated averages is a normalized block sequence in $X^p_\alpha$.  Consequently, it $1$-dominates the unit vector basis of $\ell_p$.  Fix $E_1<E_2<\ldots$, so that $E_j\in S_\alpha$ for each $j$.  Fix $(a_n)\in c_{00}$.  Let $z=\sum a_n x_n^{\alpha, I}$.  We can assume $E_j\subset I$ for each $j$ by replacing $E_j$ with $E_j\cap I$ without changing the value of $\sum\|E_jz\|_1^p$.  As before, we can decompose $x^{\alpha, I}_n = w_n+x_n+y_n$ so that $(w_n), (x_n), (y_n)$ are subnormalized block sequences, for each $n$, $\text{supp\ }w_n$ meets $E_j$ for at most one $j$, $\text{supp\ } y_n$ meets $E_j$ for at most one $j$, and for each $j$, $E_j$ meets $\text{supp\ } x_n$ for at most one $n$.  
Let $J_n = \{j: E_j\cap \text{supp\ }x_n\neq \varnothing\}$, and note that $J_1<J_2<\ldots$.  Let $x=\sum a_n x_n$.  Then $$\sum_j \|E_jx\|_1^p = \sum_{n=1}^\infty|a_n|^p \sum_{j\in J_n}\|E_jx_n\|_1^p \leq \sum_{n=1}^\infty |a_n|^p \|x_n\|_{X^p_\alpha}^p \leq \sum_{n=1}^\infty |a_n|^p.$$  
Next, let $N_j=\{n: E_j\cap \text{supp\ }w_n\neq \varnothing\}$, $w=\sum a_n w_n$.  Note that $N_1<N_2<\ldots$.  Then by Lemma \ref{c_0dom}, $$\|E_jw\|_1 \leq 2\underset{n\in N_j}{\max}|a_n| \leq 2\Bigl(\sum_{n\in N_j} |a_n|^p\Bigr)^{1/p}.$$  Therefore $$\sum_j \|E_jw\|_1^p \leq 2^p \sum_j \sum_{n\in N_j} |a_n|^p \leq 2^p\sum_n |a_n|^p.$$ 
Similarly, if $y=\sum a_ny_n$, $$\sum_j \|E_jy\|_1^p \leq 2^p \sum |a_n|^p.$$  Therefore $$\|z\|_{X_\alpha^p} \leq \|w\|_{X_\alpha^p}+\|x\|_{X_\alpha^ p}+\|y\|_{X_\alpha^p}\leq 5\Bigl(\sum |a_n|^p\Bigr)^{1/p}.$$

\end{proof}

\section{Ordinal Indices}

First, we recall the Szlenk index of a separable Banach space.  Let $X$ be a Banach space, and $K$ a weak$^*$ compact subset of $X^*$.  For $\varepsilon>0$, we define $$(K)'_\varepsilon=\Bigl\{z\in K: \text{For all\ }w^*\text{ -neighborhoods\ }U \text{\ of\ }z, \text{diam}(U\cap K)>\varepsilon\Bigr\}.$$
It is easily verified that $(K)'_\varepsilon$ is also weak$^*$ compact.  We let $$P_0(K, \varepsilon)=K$$

$$P_{\alpha+1}(K, \varepsilon)=(P_\alpha(K, \varepsilon))'_\varepsilon \text{\ \ \ }\alpha<\omega_1$$

$$P_\alpha(K, \varepsilon)=\underset{\beta<\alpha}{\bigcap}P_\beta(K, \varepsilon) \text{\ \ \ }\alpha<\omega_1,\text{\ \ }\alpha  \text{\ a limit ordinal.}$$

If there exists some $\alpha<\omega_1$ so that $P_\alpha(K, \varepsilon)=\varnothing$, we define $$\eta(K, \varepsilon)=\min \{\alpha: P_\alpha(K)=\varnothing\}.$$  Otherwise, we set $\eta(K, \varepsilon)=\omega_1$.  Then we define the Szlenk index of a Banach space $X$, denoted $Sz(X)$, to be $$Sz(X)=\underset{\varepsilon>0}{\sup}\text{\ }\eta(B_{X^*},\varepsilon).$$

The Szlenk index is one of several slicing indices.  The following two facts come from \cite{SZLENK}. \begin{enumerate}\item  For a Banach space $X$, $Sz(X)<\omega_1$ if and only if $X^*$ is separable,  \item If $X$ embeds isomorphically into $Y$, $Sz(X)\leq Sz(Y)$. \end{enumerate}
The above definition of the index is, in some cases, intractable.  A connection between weak indices and the Szlenk index has been very useful in computations.  For this, we will be concerned with a specific type of tree.  

For a Banach space $X$ and $\rho\in (0,1]$, we let $$\mathcal{H}^X_\rho=\Bigl\{(x_n)\in S_X^{<\omega}: \Bigl\|\displaystyle\sum a_nx_n\Bigr\|\geq \rho \displaystyle\sum a_n \text{\ \ \ }\forall (a_n)\subset \mathbb{R}^+\Bigr\}.$$  We will compute the Szlenk index of Baernstein spaces by combining several facts about the Szlenk index.

\begin{theorem}\cite[Theorems 3.22, 4.2]{AJO}\cite[Proposition 5]{OSZ2}
If $X$ is a Banach space such that $X^*$ is separable, there exists some ordinal $\alpha<\omega_1$ so that $Sz(X)=\omega^\alpha$. 
Moreover for any $\alpha<\omega_1$, $Sz(X)>\omega^\alpha$ if and only if there exists $\rho\in (0,1]$ and $(x_E)_{E\in \mathcal{S}_\alpha\setminus\{\varnothing\}}\subset S_X$ such that for each $E\in \mathcal{S}_\alpha\setminus MAX(\mathcal{S}_\alpha)$, $(x_{E\cup\{n\}})_{n>E}$ is weakly null and for each branch $E_1\prec E_2\prec\ldots\prec E_n$ of $\mathcal{S}_\alpha\setminus\{\varnothing\}$, $(x_{E_i})_{i=1}^n\in \mathcal{H}^X_\rho$.  \label{Alspach}
\end{theorem}

With this, we can prove the following.  

\begin{proposition} For $\alpha<\omega_1$, $p\in (1,\infty)$, $Sz(X^p_\alpha)=\omega^{\alpha+1}$. \label{BaernsteinSzlenk}\end{proposition}

\begin{proof}

Let $(e_n)$ denote the unit vector basis of $X^p_\alpha$.  For $E\in S_\alpha\setminus\{\varnothing\}$, let $x_E=e_{\max E}$.  If $E_1, \ldots, E_n$ is a branch of $S_\alpha$, then $(x_{E_i})_{i=1}^n= (e_i)_{i\in E_n}$.  Clearly $$\Bigl\|\sum_{i=1}^n a_i x_{E_i}\Bigr\|_{X_\alpha^p} = \sum_{i\in E}a_i$$ for $a_i\geq 0$.  Since the basis is normalized and shrinking, we deduce that for $E\in S_\alpha\setminus MAX(S_\alpha)$, $(x_{E\cup \{n\}})_{n>E}=(e_n)_{n>E}$ is weakly null.  Then Theorem  \ref{Alspach} guarantees $Sz(X_\alpha^p)>\omega^\alpha$.  
We must therefore only show that $Sz(X^p_\alpha)\leq \omega^{\alpha+1}$.  Suppose not.  By Theorem \ref{Alspach}, there must exist some normalized tree $(x_E)_{E\in S_{\alpha+1}\setminus \{\varnothing\}}\subset \mathcal{H}_\rho^{X^p_\alpha}$ with $x_{E\cup \{n\}}\underset{w}{\to}0$.  By standard perturbation and pruning arguments, we can assume this tree is a block tree.  For $E\in S_{\alpha+1}\setminus \{\varnothing\}$, let $m(E)=\min \text{supp}x_E$.  Because the basis is normalized, shrinking, and satisfies subsequential $4$-$X^p_\alpha$-upper block estimates in $X^p_\alpha$, we can replace $\rho$ with $\frac{\rho}{4}$ and replace the tree $(x_E)_{E\in S_{\alpha+1}\setminus\{\varnothing\}}$ with $(e_{m(E)})_{E\in S_{\alpha+1}\setminus\{\varnothing\}}$ while maintaining the two properties mentioned above.  
Choose $i_1$ so large that $5i^{1/p}_1<\frac{\rho}{16}i_1$.  Next, choose $i_2<\ldots<i_N$ such that $i_n>3i_{n-1}$ and $m(\{i_1, \ldots, i_{n-1}\})<i_n$ for each $n=2,\ldots, N$ and $E=\{i_1, \ldots, i_N\}\in MAX(\mathcal{S}_{\alpha+1})$.  Since $\mathcal{S}_{\alpha+1}$ is compact, it can contain no strictly increasing infinite chain, so such a set must exist.  Since $i_n\leq m(\{i_1, \ldots, i_n\})<i_{n+1}$, the sequence $(e_{i_n})_{n=1}^N$ $4$-dominates $(e_{m(\{i_1, \ldots, i_n\})})_{i=1}^N$.  This follows from an application of Lemma \ref{perturb} after we recall that $(e_n)$ satisfies subsequential $4$-$X^p_\alpha$-upper block estimates in $X^p_\alpha$.  Therefore for any $a_n\geq 0$, \begin{equation} \Bigl\|\sum_{i=1}^N a_n e_{i_n}\Bigr\|_{X_\alpha^p}\geq \frac{1}{4}\Bigl\|\sum_{n=1}^N a_ne_{m(\{i_1, \ldots, i_n\})}\Bigr\|_{X_\alpha^p} \geq \frac{\rho}{16}\sum a_n.\label{1}\end{equation}  
Since $E\in S_{\alpha+1}$, $\min E=i_1$, and $E\in MAX(S_{\alpha+1})$, there exist unique $E_n\in S_\alpha$ with $E_1<\ldots<E_{i_1}$ and $E=\underset{n=1}{\overset{i_1}{\bigcup}}E_n$.  Let $I=E\cup \{3^ki_N:k\in \mathbb{N}\}$.  Then $i_{n+1}\geq 3i_n$ for each $n$.  If $(x^{\alpha, I}_n)$ is the sequence of repeated averages, then $\text{supp} x^{\alpha, I}_n=E_n$ for $1\leq n \leq i_1$.  Let $a_j$ be such that $x^{\alpha, I}_n=\sum_{j\in E_n}a_j e_{j}$.  Then $\sum_{j\in E_n}a_j=1$, so \begin{equation} \sum_{n=1}^{i_1}\sum_{j\in E_n}a_j=i_1.\label{2}\end{equation}  But by Lemma \ref{lpdom}, \begin{equation}\Bigl\|\sum_{n=1}^{i_1}\sum_{j\in E_n} a_je_{j}\Bigr\|_{X_\alpha^p} =\Bigl\|\sum_{n=1}^{i_1} x_n^{\alpha,I}\Bigr\|_{X_\alpha^p} \leq 5 \Bigl\|\sum_{n=1}^{i_1} e_n\Bigr\|_{\ell_p}=5i_1^{1/p}.\label{3}\end{equation} 
Combining $(4.1), (4.2), (4.3)$, we deduce $$5i^{1/p} \geq \Bigl\|\sum_{n=1}^{i_1}\sum_{j\in E_n} a_j e_{i_j}\Bigr\|_{X^p_\alpha} \geq \frac{\rho}{16}\sum_{n=1}^{i_1}\sum_{j\in E_n}a_j = \frac{\rho}{16}i_1.$$  But this contradicts our choice of $i_1$, and completes the proof.

\end{proof}

\section{Main Theorems}

Throughout this section, $Z^{V^{(*)}}$ will denote $Z^{(V^{(*)})}$.

\begin{proof}[Proof of Lemma \ref{upperest}]    

It is clear that $(u_{c_n})\subset U^p_C$ $1$-dominates the unit vector basis of $\ell_p$.  A simple duality argument implies that $(u_{c_n}^*)\subset (U^p_C)^{(*)}$ is $1$-dominated by the unit vector basis of $\ell_{p'}$ (or $c_0$ if $p=1$), where $p'$ denotes the conjugate exponent to $p$.  Using the assumption that $1\leq p \leq 2$, we deduce that \begin{equation} \|\cdot\|_{(U^p_B)^{(U^p_C)^{(*)}}}\leq \|\cdot\|_{(U^p_B)^{p'}} \leq \|\cdot\|_{(U^p_B)^p} = \|\cdot\|_{U^p_B} \leq \|\cdot\|_{(U^p_B)^{(U^p_C)^{(*)}}}.\label{5.1}\end{equation}  It follows that all of the above norms are equal.  

Fix a normalized block sequence $(z_n)$ in $Z^{(U^p_C)^{(*)}}$.  Let $m_n=\min \text{supp\ }z_n$.  Fix $(a_n)\in c_{00}$, $z=\sum a_nz_n$, and integers $1= k_0<k_1<\ldots$ so that $$\|z\|_{Z^{(U^p_C)^{(*)}}}=\Bigl\|\sum_{i=1}^\infty \|P^E_{[k_{i-1}, k_i)}z\|_Z u^*_{c_{k_{i-1}}}\Bigr\|_{(U^p)^{(*)}}.$$

We can write $z_n=w_n+x_n+y_n$ so that for each $n$, $w_n,x_n,y_n$ are either $0$ or projections of $z_n$ onto intervals, there is at most one $i$ so that $[k_{i-1}, k_i)\cap \text{supp\ }w_n\neq \varnothing$, the same holds for $y_n$, and for each $i$, there is at most one $n$ so that $[k_{i-1}, k_i)\cap \text{supp\ }x_n\neq \varnothing$.  Then $(w_n),(x_n),(y_n)$ are subnormalized block sequences in $Z^{(U^p_C)^{(*)}}$ by bimonotonicity, as long as we omit zero terms from the sequence.  Because $\|\cdot\|_Z\leq \|\cdot\|_{Z^{(U^p_C)^{(*)}}}$, these sequences are also subnormalized in $Z$.  Let $w=\sum a_nw_n$, and define $x,y$ similarly.  

Let $$N_i=\{n: [k_{i-1}, k_i)\cap \text{supp\ }w_n\neq \varnothing\},$$ $$I=\{i: N_i \neq \varnothing\},\text{\ \ \ } N=\underset{i\in I}{\bigcup}N_i.$$  If $I$ is empty, then $w_n=0$ for all $n$, and $w=0$.  Assume $I\neq \varnothing$.  By construction, $N_1<N_2<\ldots$.  Moreover, $(m_n)_{n\in N_i},(\min \text{supp\ }x_n)_{n\in N_i}$ are ordered so that we can apply Lemma $2.2$.  Therefore $$\Bigl\|P^E_{[k_{i-1}, k_i)}w\Bigr\|_Z \leq K\Bigl\|\sum_{n\in N_i} \|w_n\|_Za_n u_{b_{m_n}}\Bigr\|_{U^p} \leq K\Bigl\|\sum_{n\in N_i}a_nu_{b_{m_n}}\Bigr\|_{U^p},$$ where $E$ satisfies subsequential $K$-$U^p_B$-upper block estimates in $Z$.  We deduce \begin{align*} \Bigl\|\sum_{i=1}^\infty \|P^E_{[k_{i-1}, k_i)}w\|_Z u^*_{c_{k_{i-1}}}\Bigr\|_{(U^p)^{(*)}} & = \Bigl\|\sum_{i\in I} \|P^E_{[k_{i-1},k_i)}w\|_Z u^*_{c_{k_{i-1}}}\Bigr\|_{(U^p)^{(*)}} \\ & \leq K\Bigl\|\sum_{i\in I}\Bigl\|\sum_{n\in N_i} a_n u_{b_{m_n}}\Bigr\|_{U^p} u^*_{c_{k_{i-1}}}\Bigr\|_{(U^p)^{(*)}} \\ & = K\Bigl\|\sum_{n\in N} a_n u_{b_{m_n}}\Bigr\|_{U^p} \leq K\Bigl\|\sum_{n=1}^\infty a_n u_{b_{m_n}}\Bigr\|_{U^p}.\end{align*}    
The last equality follows from \eqref{5.1}.  A similar argument gives $$\Bigl\|\sum_{i=1}^\infty \|P^E_{[k_{i-1}, k_i)}y\|_Z u^*_{c_{k_{i-1}}}\Bigr\|_{(U^p)^{(*)}}\leq K\Bigl\|\sum_{n=1}^\infty a_n u_{b_{m_n}}\Bigr\|_{U^p}.$$ 

Next, let $$I_n=\{i: [k_{i-1}, k_i)\cap \text{supp\ }x_n\neq \varnothing\},$$ $$N'=\{n: I_n\neq \varnothing\}, \text{\ \ \ }I'=\underset{n\in N'}{\bigcup}I_n.$$  Note that $I_1<I_2<\ldots$.  If $N'=\varnothing$, then $x=0$.  Assume $N'\neq \varnothing$.  For $n\in N'$, let $$v_n = \sum_{i\in I_n} \|P^E_{[k_{i-1}, k_i)}x_n\|_Z u_{c_{k_{i-1}}}^*.$$  Then $$\|v_n\|_{(U^p)^{(*)}} \leq \|x_n\|_{Z^{(U^p_C)^{(*)}}} \leq 1.$$  Moreover, $$\sum_{i\in I_n}\|P^E_{[k_{i-1}, k_i)}x\|_Zu_{c_{k_{i-1}}}^* = |a_n|\sum_{i\in I_n}\|P^E_{[k_{i-1},k_i)}x_n\|_Z u^*_{c_{k_{i-1}}} = |a_n|v_n.$$  Therefore \begin{align*} \Bigl\|\sum_{i=1}^\infty \|P^E_{[k_{i-1}, k_i)}x\|_Z u^*_{c_{k_{i-1}}} \Bigr\|_{(U^p)^{(*)}} & = \Bigl\|\sum_{n\in N'}\sum_{i\in I_n} \|P^E_{[k_{i-1}, k_i)}x\|_Z u^*_{c_{k_{i-1}}}\Bigr\|_{(U^p)^{(*)}} \\ & = \Bigl\|\sum_{n\in N'} |a_n|v_n\Bigr\|_{(U^p)^{(*)}} \leq \Bigl\|\sum_{n\in N'} a_ne_n\Bigr\|_{p'} \\ & \leq \Bigl\|\sum_{n\in N'} a_ne_n\Bigr\|_p \leq \Bigl\|\sum_{n=1}^\infty a_n u_{b_{m_n}}\Bigr\|_{U^p}.\end{align*}  
We deduce \begin{align*} \|z\|_{Z^{(U^p)^{(*)}}} & =\Bigl\|\sum_{i=1}^\infty \|P^E_{[k_{i-1}, k_i)}z\|_Z u^*_{c_{k_{i-1}}} \Bigr\|_{(U^p)^{(*)}} \\ & \leq  \Bigl\|\sum_{i=1}^\infty \|P^E_{[k_{i-1}, k_i)}w\|_Z u^*_{c_{k_{i-1}}} \Bigr\|_{(U^p)^{(*)}}  \\ & +  \Bigl\|\sum_{i=1}^\infty \|P^E_{[k_{i-1}, k_i)}x\|_Z u^*_{c_{k_{i-1}}} \Bigr\|_{(U^p)^{(*)}}  \\ & + \Bigl\|\sum_{i=1}^\infty \|P^E_{[k_{i-1}, k_i)}y\|_Z u^*_{c_{k_{i-1}}} \Bigr\|_{(U^p)^{(*)}} \\ & \leq (2K+1)\Bigl\|\sum_{n=1}^\infty a_n u_{b_{m_n}}\Bigr\|_{U^p}. \end{align*}  Therefore $E$ satisfies subsequential $(2K+1)$-$U^p_B$-upper block estimates in $Z^{(U^p_C)^{(*)}}$.

\end{proof}

Our first major theorem generalizes \cite[Theorem 15]{OSZ1}.   

\begin{theorem}  Let $U$ be a Banach space with a normalized, shrinking, $1$-unconditional, right dominant basis $(u_n)$ satisfying subsequential $U^p$-upper block estimates in $U^p$, where $1<p\leq 2$.  If $X$ is a separable, reflexive Banach space which satisfies subsequential $((U^p)^*,U^p)$-tree estimates, then $X$ embeds in a reflexive Banach space $\tilde{X}$ with bimonotone FDD $E$ satisfying subsequential $((U^p)^*,U^p)$-upper block estimates.  \label{main1}\end{theorem}

\begin{proof}

Throughout the proof, we will repeatedly use \cite[Corollaries 7,9]{OSZ1} to deduce the reflexitivity of given spaces, beginning with $U^p$.  Since $X$ satisfies subsequential $U^p$-upper tree estimates, \cite[Proposition 4]{OSZ1} implies that $X^*$ satisfies subsequential $(U^p)^*$-lower tree estimates.  By \cite[Theorem 12(b)]{OSZ1}, there exists a Banach space $Y$ with bimonotone shrinking FDD $F$ and $M\in [\mathbb{N}]^\omega$ such that $X^*$ is a quotient of $Z=Y^{(U^p_M)^*}(F)$.  By \cite[Lemma 2.11]{CA}, $F$ satisfies subsequential $(U^p_M)^*$-lower block estimates in $Z$.  By \cite[Lemma 2.13]{CA}, the space $W=Z\oplus (U^p_{\mathbb{N}\setminus M})^*$ has a bimontone FDD $G$ satisfying subsequential $(U^p)^*$-lower block estimates.  Then $X^*$ is a quotient of $W$, and $W$ is reflexive.  By duality, $X$ is a subspace of $W^*$, which is reflexive with bimonotone FDD $G^*=(G_n^*)$ satisfying subsequential $U^p$-upper block estimates in $W^*$.  

By \cite[Theorem 12(a)]{OSZ1}, there exists a blocking $H$ of $G^*$ defined by $H_k = \underset{i=b_k}{\overset{b_{k+1}-1}{\oplus}} G_i^*$ for some $1=b_1<b_2<\ldots$ and $C\in [\mathbb{N}]^\omega$ so that $X\hookrightarrow (W^*)^{(U^p_C)^*}(H)$.  We deduce from the fact that $G^*$ satisfies subsequential $U^p$-upper block estimates in $W^*$ that $H$ satisfies subsequential $U_B^p$-upper block estimates in $W^*$.  Let $k_i= \max\{b_i, c_i\}$.  Since $(u_n)$ is right dominant, $H$ satisfies subsequential $U^p_K$-upper block estimates in $W^*$.  Lemma \ref{upperest} implies that $H$ satisfies subsequential $U^p_K$-upper block estimates in $(W^*)^{(U^p_C)^*}(H)$.  By \cite[Lemma 2.11]{CA}, $H$ satisfies subsequential $(U^p_C)^*$-lower block estimates in $(W^*)^{(U^p_C)^*}(H)$, and since $(u^*_n)$ is left dominant, $H$ satisfies subsequential $(U^p_K)^*$-lower block estimates in $(W^*)^{(U^p_C)^*}(H)$.  By the proof of \cite[Lemma 2]{OSZ1}, we deduce $\tilde{X} = (W^*)^{(U^p_C)^*}(H)\oplus (U^p_{\mathbb{N}\setminus K})^*$ has bimonotone FDD satisfying subsequential $((U^p)^*,U^p)$-block estimates in $\tilde{X}$.  Again, \cite[Corollaries 7,9]{OSZ1} guarantee that $\tilde{X}$ is reflexive, and this completes the proof.

\end{proof}

\begin{remark}  The statement of \cite[Lemma 2]{OSZ1} assumed the basis $(u_n)$ has a stronger property called \emph{block stability}, which is easily seen to be equivalent to the basis satisfying subsequential $U^p$-upper and $U^p$-lower block estimates in $U^p$.  One can easily check that the canonical basis of $\ell_p$ does not dominate any subsequence of $X_\alpha^p$ except in the trivial cases $p=1$ or $\alpha=0$.  Since $X^p_\alpha$ contains normalized block sequences equivalent to the canonical $\ell_p$ basis, we see the weakening of the block stability assumption is necessary.  The proof in \cite{OSZ1}, however, only uses the fact that $(u_n), (u_n^*)$ satisfy subsequential $U^p$-upper block estimates and $(U^p)^*$-lower block estimates in $U^p$ and $(U^p)^*$, respectively.  We have the upper block estimates by hypothesis and \eqref{5.1}, while the lower block estimates follow by duality.

\end{remark}

The next theorem is a generalization of \cite[Theorem 21]{OSZ1}, and an adaptation of \cite[Theorem 5.4]{CA} to our situation.  Let us recall that if $U, V$ are Banach spaces with a normalized, $1$-unconditional bases, $\mathcal{A}_{(V,U)}$ denotes the class of separable, reflexive Banach spaces satisfying subsequential $(V,U)$-tree estimates.  

\begin{theorem} Let $U$ be a Banach space with basis satisfying the hypotheses of Theorem \ref{main1}, and let $1<p\leq 2$.  Then the class $\mathcal{A}_{((U^p)^*,U^p)}$ contains a reflexive universal element with bimonotone FDD.  \label{main2}\end{theorem}

\begin{proof}

Fix constants $R,K$ so that $(u_n)$ is $R$-right dominant and satisfies subsequential $K$-$U^p$-upper block estimates in $U^p$.  

By a result of Schechtman \cite{Sch}, there exists a Banach space $W$ with bimontone FDD $E=(E_n)$ with the property that any Banach space with bimonotone FDD embeds almost isometrically into $\overline{\underset{n=1}{\overset{\infty}{\oplus}}E_{k_n}}$ for some subsequence $(k_n)$ of the natural numbers, and this subspace is $1$-complemented in $W$.  More precisely, given a Banach space with bimonotone FDD $(F_i)$ and $\varepsilon>0$, there is a subsequence $(E_{k_n})$ of $(E_n)$ and a $(1+\varepsilon)$-embedding $T:X\to W$ such that $T(F_n)=E_{k_n}$ for all $n\in \mathbb{N}$, and $\sum_{n=1}^\infty P^E_{k_n}$ is a norm-$1$ projection of $W$ onto $\overline{\underset{n=1}{\overset{\infty}{\oplus}}E_{k_n}}$.  

We next consider the space $W_0=(W^{(*)})^{(U^p)^*}$.  By \cite[Corollary 7]{OSZ1}, the sequence $(E^*_n)$ is a boundedly complete and bimonotone FDD for this space.  This means that $W_0=(W_0^{(*)})^*$ and $(E_n^{**})=(E_n)$ is a shrinking, bimonotone FDD for $W_0^{(*)}$.  Therefore $W_0$ is naturally the dual of the space $Y=W_0^{(*)}$ with bimonotone shrinking FDD $E$.  By duality and \cite[Lemma 2.11]{CA}, we deduce that $E$ satisfies subsequential $2K$-$U^p$-upper block estimates in $Y$.

Let $Z=Y^{(U^p)^*}$.  By Lemma \ref{upperest}, $E$ satisfies subsequential $4K+1$-upper block estimates in $Z$.  By \cite[Lemma 2.11]{CA}, $E$ satisfies subsequential $2K$-$U^p$-lower block estimates in $Z$.  By \cite[Corollary 7, Lemma 8]{OSZ1}, $E$ is a shrinking, boundedly-complete FDD for $Z$.  Therefore $Z\in \mathcal{A}_{((U^p)^*,U^p)}$.  We see also that $E$ is bimonotone in $Z$.  It remains to show the universality of $Z$ for $\mathcal{A}_{((U^p)^*,U^p)}$.  

Let $D\geq 1$ and assume $X$ satisfies subsequential $D$-$((U^p)^*,U^p)$-tree estimates.  By Theorem \ref{main1}, there exists a reflexive Banach space $\tilde{X}$ with bimonotone FDD $D$ satisfying subsequential $((U^p)^*,U^p)$ block estimates in $\tilde{X}$ so that $X$ embeds isomorphically into $\tilde{X}$.  Thus it suffices to assume $X$ itself has a bimonotone FDD $F$ satisfying subsequential $D_1$-$((U^p)^*,U^p)$ block estimates and show that $X$ embeds into $Z$.  We can find a subsequence $(k_n)$ of $\mathbb{N}$ and a $2$-embedding $T:X\to W$ so that $T(F_n)=E_{k_n}$ for all $n\in \mathbb{N}$ and $\sum_n P^E_{k_n}$ is a norm-$1$ projection of $W$ onto $\overline{\oplus_n E_{k_n}}$.  It follows that $(E_{k_n})$ satisfies subsequential $2D_1$-$((U^p)^*,U^p)$ estimates in $W$.  By duality, $(E^*_{k_n})$ satisfies subsequential $((U^p)^*, U^p)$ estimates in $W^{(*)}$.  We will last prove that the norms $\|\cdot\|_W, \|\cdot\|_Y, \|\cdot\|_Z$ are equivalent when restricted to $c_{00}\bigl(\oplus_n E_{k_n}\bigr)$.  

Fix $w^*\in c_{00}\bigl(\oplus_n E_{k_n}\bigr)$.  We know $\|w^*\|_{W^{(*)}}\leq \|w^*\|_{Y^*}$.  Choose $1\leq m_0<m_1<\ldots<m_N$ in $\mathbb{N}$ such that $$\|w^*\|_{Y^*} = \Bigl\|\sum_{n=1}^N \|P^{E^*}_{[m_{n-1}, m_n)}w^*\|_{W^{(*)}} u^*_{m_{n-1}}\Bigr\|_{(U^p)^*}.$$  By discarding any $m_n$ so that $P^{E^*}_{[m_{n-1},m_n)}w^*=0$, we assume $P^{E^*}_{[m_{n-1}, m_n)}w^*\neq 0$ for each $1\leq n \leq N$ without changing the sum.  There exist $j_1<\ldots<j_N$ so that $m_n>k_{j_n}= \min \text{supp}_{E^*} P^{E^*}_{[m_{n-1},m_n)}w^*\geq m_{n-1}$ for each $1\leq n \leq N$.  Since $(u_n^*)$ satisfies subsequential $K$-lower block estimates in $(U^p)^*$ and is $R$-left dominant, and since $(E^*_{k_n})$ satisfies subsequential $2D_1$-$(U^p)^*$-lower block estimates in $W^{(*)}$, we see \begin{align*} \|w^*\|_{Y^*} & \leq K\Bigl\|\sum_{n=1}^N \|P^{E^*}_{[m_{n-1}, m_n)}w^*\|_{W^{(*)}} u^*_{k_{j_{n-1}}}\Bigr\|_{(U^p)^*} \\ & \leq KR \Bigl\|\sum_{n=1}^N \|P^{E^*}_{[m_{n-1}, m_n)}w^*\|_{W^{(*)}} u^*_{j_n}\Bigr\|_{(U^p)^*} \\ & \leq 2KRD_1 \|w^*\|_{W^{(*)}}. \end{align*}

This shows $\|\cdot\|_{W^{(*)}}$ and $\|\cdot\|_{Y^*}$ are equivalent on $c_{00}\bigl(\sum_n E^*_{k_n}\bigr)$.  One easily sees that $\sum_n P^{E^*}_{k_n}$, which defines a norm-$1$ projection of $W^{(*)}$ onto $\overline{\oplus_n E^*_{k_n}}$, is also a norm-$1$ projection of $Y^*$ onto $\overline{\oplus_n E^*_{k_n}}$.  It follows that $$\frac{1}{2KRD_1}\|w\|_W\leq \|w\|_Y\leq \|w\|_W$$ for all $w\in c_{00}\bigl(\sum_n E_{k_n}\bigr)$.  

A very similar argument shows that $\|y\|_Y\leq \|y\|_Z\leq 2KRD_1\|y\|_Y$ for each $y\in c_{00}\bigl(\sum_n E_{k_n}\bigr)$.  Therefore the map $T:X\to W$ becomes an $8K^2R^2D_1^2$ embedding of $X$ into $Z$.

\end{proof}

For our next theorem, we define for an ordinal $\alpha<\omega_1$ $$\mathcal{C}_\alpha=\{X: X\text{separable, reflexive}, Sz(X), Sz(X^*)\leq \omega^\alpha\}.$$  
\begin{theorem} For any $\alpha<\omega_1$ and $p\in (1,2]$, there exists a Banach space $Z\in\mathcal{C}_{\alpha+1}$ with bimonotone FDD satisfying subsequential $((X_\alpha^p)^*, X^p_\alpha)$-block estimates such that if $X\in \mathcal{C}_\alpha$, $X$ is isomorphic to a subspace of $Z$.  Moreover, there exists $W\in \mathcal{C}_{\alpha+1}$ with a basis such that if $X\in \mathcal{C}_\alpha$, $X$ is isomorphic to a subspace of $W$. \label{main3}\end{theorem}

\begin{proof}

Let $Z$ be the universal element of $\mathcal{A}_{((X^p_\alpha)^*, X^p_\alpha)}$ guaranteed by Theorem \ref{main2}.  Then $Z, Z^*$ satisfy subsequential $X^p_\alpha$-upper block estimates, and $Sz(Z), Sz(Z^*)\leq Sz(X^p_\alpha) = \omega^{\alpha+1}$ by \cite[Corollary 4.5]{CA}.  Therefore $Z\in \mathcal{C}_{\alpha+1}$.  By \cite[Corollary 4.12]{JRZ}, there exists a sequence of finite dimensional spaces $(H_n)$ so that if $D=\Bigl(\underset{n=1}{\overset{\infty}{\oplus}} H_n\Bigr)_2$, then $W=Z\oplus_2 D$ is reflexive and has a basis.  Since the FDD $(H_n)$ satisfies $\ell_2$-upper block estimates in $D$, $Sz(D)\leq \omega$ \cite[Theorem 3]{OS}.  By \cite[Proposition 14]{OSZ2}, $$Sz(W)= \max \{Sz(Z), Sz(D)\}\leq \omega^{\alpha+1}.$$  By the same reasononing, $Sz(W^*)=Sz(Z^*\oplus_2 D^*)\leq \omega^{\alpha+1}$.  Therefore $W\in \mathcal{C}_{\alpha+1}$.  

If $X\in \mathcal{C}_\alpha$, \cite[Theorem 1.1]{CA} implies that $X,X^*$ both satisfy subsequential $X_\alpha$-upper tree estimates, and therefore also satisfy $X_\alpha^p$-upper tree estimates.  Then by \cite[Lemma 2.7]{FOSZ}, $X$ satisfies subsequential $((X_\alpha^p)^*,X_\alpha^p)$-tree estimates.  By universality of $Z$, $X$ embeds isomorphically in $Z$, and therefore $X$ embeds isomorphically into $W$.

\end{proof}

\end{document}